\documentclass{amsart}
\usepackage{amssymb,amsfonts,amsthm,amsaddr,graphicx,fullpage}
\usepackage{color}
\begin{document}
\newcommand{\tr}{{\mathrm{tr}}}
\newcommand{\bM}{{\mathbf{M}}}
\newcommand{\sgn}{{\mathrm{sgn}}}
\newcommand{\modndx}{\gamma_\mathrm{M}}
\newcommand{\parndx}{\gamma_\mathrm{P}}

\newtheorem{theorem}{Theorem}
\newtheorem{corollary}{Corollary}
\newtheorem{lemma}{Lemma}

\theoremstyle{definition}
\newtheorem{definition}{Definition}

\title[Dynamical Hamiltonian-Hopf Instabilities of Periodic Traveling Waves]{Dynamical Hamiltonian-Hopf Instabilities of Periodic Traveling Waves in Klein-Gordon Equations}
\author{R. Marangell}
\address{School of Mathematics and Statistics\\University of Sydney}
\email{robert.marangell@sydney.edu.au}
\author{P. D. Miller}
\address{Department of Mathematics\\University of Michigan}
\email{millerpd@umich.edu}
\date{\today}
\begin{abstract}
We study the unstable spectrum close to the imaginary axis for the linearization of the nonlinear Klein-Gordon equation about a periodic traveling wave in a co-moving frame. We define dynamical Hamiltonian-Hopf instabilities as points in the stable spectrum that are accumulation points for unstable spectrum, and show how they can be determined from the knowledge of the discriminant of an associated Hill's equation.  This result allows us to give simple criteria for the existence of dynamical Hamiltonian-Hopf instabilities in terms of instability indices previously shown to be useful in stability analysis of periodic traveling waves.
\end{abstract}
\maketitle

\section{Introduction}
The Klein-Gordon equation
\begin{equation}
\frac{\partial^2 u}{\partial t^2}-\frac{\partial^2 u}{\partial x^2} + V'(u)=0,
\label{eq:KG}
\end{equation}
where $V:\mathbb{R}\to\mathbb{R}$ is a $C^2$ potential, is one of the most ubiquitous models for nonlinear wave motion in one space dimension.  Perhaps the earliest physical application of this equation is its use by Gordon and Klein \cite{Gordon26,Klein27}, in the case where $V$ is a linear function of $u$, as a model for the relativistic motion of charged particles with effective mass proportional to $\sqrt{V'}$.  The case of $V$ nonlinear had actually occurred earlier (and in different, characteristic coordinates) in the geometry of surfaces of constant negative curvature, see for example, \cite{Eisenhart09}.  The latter application corresponds to the special case of $V(u)=-\cos(u)$,
in which \eqref{eq:KG} is known as the sine-Gordon equation.  The sine-Gordon equation has seen physical applications ranging from the modeling of vibrations of DNA molecules through quantum field theory, and for a thorough overview we refer to the review paper of Barone et al., \cite{BaroneEMS71}.

This note is concerned with periodic traveling wave solutions of nonlinear Klein-Gordon equations and their linearized stability properties.  Of particular interest to us here will be the behavior of the spectrum near, but not on, the imaginary axis in the complex plane.
\subsection{Linear stability analysis of periodic traveling waves} 
Let  $c\neq \pm 1$ be a constant velocity, and consider a traveling wave solution $u(x,t)=f(x-ct)$
of the nonlinear Klein-Gordon equation \eqref{eq:KG}.  The wave profile $f$ satisfies the ordinary differential equation
\begin{equation}
(c^2-1)f''(z)+V'(f(z))=0.
\label{eq:KG-traveling-wave}
\end{equation}
Under suitable conditions on $V$ (for example if $V$ has at least one local extremum)\footnote{In this paper, we have followed the setup and notation used in \cite{JonesMMP14}. While we intend for the current paper to stand alone (and in particular the spectral features considered here are not addressed in \cite{JonesMMP14}), we refer the reader therein for a detailed account of periodic traveling waves for the nonlinear Klein-Gordon equation and many of their stability properties.}, there will exist traveling waves for which $f'(z)$ is a periodic function of $z=x-ct$ with fundamental period $T$.  We call such solutions \emph{periodic traveling waves}.  

It follows from \eqref{eq:KG-traveling-wave} that $V^{(p)}(f(z))$ is a periodic function for all $p=0,1,2,\dots$ (the fundamental period of which may be an integer fraction of $T$).  The linearized stability properties of the periodic traveling wave $u=f(x-ct)$ are determined from the spectral problem
\begin{equation}
(c^2-1)w''(z)-2c\lambda w'(z)+(\lambda^2+V''(f(z)))w(z)=0,\quad\lambda\in\mathbb{C}.
\label{eq:spectral-problem}
\end{equation}
Indeed, every solution of \eqref{eq:spectral-problem} gives rise to a solution of the linearized Klein-Gordon equation of the form $w(x-ct)e^{\lambda t}$, for which $\lambda\in\mathbb{C}$ is the exponential growth rate.
Equation \eqref{eq:spectral-problem} is a linear ordinary differential equation with periodic coefficients of period $T$.
The (Floquet) \emph{spectrum} of this problem is defined as the set $\sigma\subset\mathbb{C}$ of values of $\lambda$ for which \eqref{eq:spectral-problem} has a nontrivial solution $w:\mathbb{R}\to\mathbb{C}$ that is a bounded function.  Equivalently, $\lambda\in\sigma$ if and only if there exists a nontrivial solution $w(z)$ of Bloch form:  $w(z)=e^{i\theta z/T}W(z)$, where $\theta\in\mathbb{R}$ and $W(z+T)=W(z)$.  The phase factor $e^{i\theta}$ picked up by the Bloch solution as $z$ increases through a period is called the \emph{Floquet multiplier} of the solution and is denoted $\mu=\mu(\lambda)$.  More generally, for each $\lambda\in\mathbb{C}$ one may define two Floquet multipliers $\mu=\mu(\lambda)$ as the eigenvalues of the (entire) \emph{monodromy matrix} \begin{equation}
\mathbf{M}(\lambda):=\begin{bmatrix} w_1(T;\lambda) & w_2(T;\lambda)\\w_1'(T;\lambda) & w_2(T;\lambda)\end{bmatrix},
\end{equation}
where $w_j(z;\lambda)$, $j=1,2$ are the unique solutions of \eqref{eq:spectral-problem} satisfying the initial conditions
\begin{equation}
w_1(0;\lambda)=w_2'(0;\lambda)=1\quad\text{and}\quad
w_1'(0;\lambda)=w_2(0;\lambda)=0.
\label{eq:canonical-IC}
\end{equation}
The spectrum $\sigma$ may then equivalently be characterized as the set of $\lambda\in\mathbb{C}$ for which at least one Floquet multiplier has unit modulus, and hence may be written in the form $\mu(\lambda)=e^{i\theta}$ for some $\theta\in\mathbb{R}$. Because the Klein-Gordon equation \eqref{eq:KG} is a real Hamiltonian system, $\sigma$ is symmetric with respect to reflection through the real and imaginary axes (see \cite[Section 3.3.1]{JonesMMP14} for details). Symmetry with respect to the imaginary axis means that growth rates $\lambda$ with positive real parts are always paired with growth rates with negative real parts.  For this reason we say that if $\sigma\subset i\mathbb{R}$ then the periodic traveling wave $u=f(x-c t)$ is said to be \emph{spectrally stable} and otherwise is \emph{spectrally unstable}.  More generally, we call $\sigma\cap i\mathbb{R}$ the \emph{stable spectrum} and $\sigma\cap(\mathbb{C}\setminus i\mathbb{R})$ the \emph{unstable spectrum}.
The spectrum $\sigma$ is a closed set consisting of a union of smooth arcs.

The spectral problem \eqref{eq:spectral-problem} is not a standard eigenvalue problem as the spectral parameter $\lambda$ enters both linearly and quadratically.   Consider however the substitution suggested by Scott \cite{Scott69}:  $w(z)=e^{c\lambda z/(c^2-1)}y(z)$, which transforms \eqref{eq:spectral-problem} into \emph{Hill's equation}:
\begin{equation}
y''(z)+P(z)y(z)=\nu y(z),\quad P(z):=\frac{V''(f(z))}{c^2-1},\quad \nu=\nu(\lambda):=\left(\frac{\lambda}{c^2-1}\right)^2.
\label{eq:Hill}
\end{equation}
The natural spectral parameter is denoted $\nu\in\mathbb{C}$.  We may consider the well-defined basis $y_j(z;\nu)$, $j=1,2$ of solutions of \eqref{eq:Hill} satisfying initial conditions of the form \eqref{eq:canonical-IC}.
In a completely analogous way, we may then define the (entire) monodromy matrix for Hill's equation as
\begin{equation}
\mathbf{M}^\mathrm{H}(\nu):=\begin{bmatrix}y_1(T;\nu) & y_2(T;\nu)\\y'_1(T;\nu) & y_2'(T;\nu)
\end{bmatrix}
\end{equation}
having Hill Floquet multipliers $\mu^\mathrm{H}(\nu)$ as its eigenvalues.  The spectrum $\Sigma^\mathrm{H}$ of Hill's equation is the set of $\nu\in\mathbb{C}$ for which one of the multipliers $\mu^\mathrm{H}(\nu)$ has unit modulus.  Because $\det(\mathbf{M}^\mathrm{H}(\nu))=1$ by Abel's Theorem, the Hill Floquet multipliers are completely determined by the trace of the Hill monodromy matrix alone, which is typically (see e.g. \cite{MagnusW04}) called the \emph{Hill discriminant}:
\begin{equation}
\Delta^\mathrm{H}(\nu):=\tr(\mathbf{M}^\mathrm{H}(\nu)).
\end{equation}
The spectrum $\Sigma^\mathrm{H}$ thus corresponds to the condition that $\Delta^\mathrm{H}(\nu)\in [-2,2]$.  It is well-known \cite{MagnusW04} that $\Sigma^\mathrm{H}$ is a closed subset of the real $\nu$-axis (as the Floquet spectral problem for Hill's equation \eqref{eq:Hill} is selfadjoint), that it is bounded above but unbounded below, and that it consists of a disjoint union of closed intervals, the endpoints of which constitute the simple periodic and antiperiodic spectrum for \eqref{eq:Hill}.

The Hill spectrum $\Sigma^\mathrm{H}$ may be pulled back to the $\lambda$-plane via the relation $\nu=\nu(\lambda)$ given in \eqref{eq:Hill}, resulting in a set we denote $\sigma^\mathrm{H}$.  A key point is that in general $\sigma\neq\sigma^\mathrm{H}$ due to the fact that unless $\lambda$ is purely imaginary the exponential factor in Scott's substitution is not bounded on $\mathbb{R}$.  In fact it is shown in \cite{JonesMMP14} that $\sigma\cap\sigma^\mathrm{H}=\sigma^\mathrm{H}\cap i\mathbb{R}$.

\subsection{Dynamical Hamiltonian-Hopf instabilities}
This note is concerned with the analysis of $\sigma$ near the imaginary axis in the $\lambda$-plane.
In particular, we wish to characterize those $\lambda\in i\mathbb{R}$ near which there may be unstable spectrum.  Consider the following definition.
\begin{definition}[dynamical Hamiltonian-Hopf instabilities]
Let $\lambda$ be a nonzero imaginary number.
The periodic traveling wave $u=f(x-ct)$ is said to exhibit a dynamical Hamiltonian-Hopf instability at $\lambda$ if for every neighborhood $U$ of $\lambda$, $\sigma\cap U\setminus i\mathbb{R}\neq\emptyset$.
\end{definition}
Thus $f$ has a dynamical Hamiltonian-Hopf instability at $\lambda\in i\mathbb{R}$ if there is unstable spectrum arbitrarily close to $\lambda$.  Since $\sigma$ is closed, it is neccessary that $\lambda\in\sigma$ for $f$ to exhibit a dynamical Hamiltonian-Hopf instability at $\lambda\in i\mathbb{R}$.  Moreover, since $\sigma$ and $\sigma^\mathrm{H}$ agree on the imaginary axis, and since $\sigma^\mathrm{H}\cap i\mathbb{R}$ consists of a union of closed imaginary intervals, a point $\lambda\in i\mathbb{R}$ of dynamical Hamiltonian-Hopf instability simultaneously is contained in an imaginary interval of stable spectrum and is an accumulation point of non-imaginary unstable spectrum.  This local structure of $\sigma$ is therefore reminiscent of the paths taken by eigenvalues of a structurally unstable system undergoing a Hamiltonian-Hopf bifurcation, which explains the terminology.  Note, however, that a system exhibiting a {\em dynamical} Hamiltonian-Hopf instability can be structurally stable.

Let $C\subset\mathbb{R}$ denote the discrete set of real critical points of the Hill discriminant $\Delta^\mathrm{H}(\nu)$, and consider the 
function $F:\mathbb{R}\setminus C\to\mathbb{R}$  defined in terms of the Hill discriminant by
\begin{equation}
F(\nu):=-c^2T^2\frac{4-\Delta^\mathrm{H}(\nu)^2}{4\Delta_\nu^\mathrm{H}(\nu)^2},\quad
\Delta_\nu^\mathrm{H}(\nu):=\frac{d\Delta^\mathrm{H}}{d\nu}(\nu).
\label{eq:F-define}
\end{equation}
While $F(\nu)$ is not defined for $\nu=\nu_0\in C$, we may extend the definition to include the critical points as follows.  If $4-\Delta^\mathrm{H}(\nu_0)^2\neq 0$, then $F(\nu)$ has a definite sign $\pm$ in the limit $\nu\to\nu_0\in C$ and we set $F(\nu_0)=\pm\infty$.  On the other hand, the Hill discriminant $\Delta^\mathrm{H}(\cdot)$ has the property that all real roots of $4-\Delta^\mathrm{H}(\nu)^2$ are either simple or double roots \cite[Lemma 2.5]{MagnusW04}.  Therefore, if
$\nu_0\in C$ and $4-\Delta^\mathrm{H}(\nu_0)^2=0$, then by l'H\^{o}pital's rule $F(\nu)$ has a finite limit as $\nu\to \nu_0$, and we may define $F$ by continuity for such $\nu_0\in C$ (these are precisely the \emph{double points} of the Hill spectrum $\Sigma^\mathrm{H}$).
Our main result is the following.
\begin{theorem}
The periodic traveling wave $u=f(x-ct)$ exhibits a dynamical Hamiltonian-Hopf instability at a nonzero point $\lambda\in i\mathbb{R}$ if and only if $\nu(\lambda)=F(\nu(\lambda))$ where $\nu(\lambda)$ is defined in \eqref{eq:Hill}.
\label{theorem-HH-characterize}
\end{theorem}
Note that if $\nu\le 0$ and $\nu=F(\nu)$, then from \eqref{eq:F-define} it follows that $\Delta^\mathrm{H}(\nu)^2\le 4$, implying that $\nu\in\Sigma^\mathrm{H}$.  Hence the corresponding values $\lambda$ lie in $\sigma^\mathrm{H}\cap i\mathbb{R}=\sigma\cap i\mathbb{R}$.

We give the proof of Theorem~\ref{theorem-HH-characterize} in \S\ref{sec:local-analysis} below.  The proof actually gives more detailed information about the spectrum $\sigma$ near a dynamical Hamiltonian-Hopf instability point on the imaginary axis; in particular in each sufficiently small neighborhood $U$ of a dynamical Hamiltonian-Hopf instability point $\lambda$, $\sigma\cap U$ consists of $U\cap i\mathbb{R}$ and a nonzero number of smooth curves crossing the imaginary axis transversely at $\lambda$ (the system of curves has to be symmetric with respect to reflection through the imaginary axis by Hamiltonian symmetry)  In particular, there cannot be any unstable spectrum that is tangent to the imaginary axis at $\lambda$.
This is to be contrasted with the behavior of the spectrum $\sigma$ near the origin in which case unstable spectrum can be tangent to the imaginary axis (weak modulational instability) or not (strong modulational instability) \cite[Definition 6.14]{JonesMMP14}.  Adapting this terminology, we may say that all dynamical Hamiltonian-Hopf instabilities are necessarily strong.  

One implication of Theorem~\ref{theorem-HH-characterize} is the following.
\begin{corollary}
Let the nonzero value $\lambda\in\sigma\cap i\mathbb{R}$ be such that $\nu(\lambda)$ is the endpoint of an interval of $\Sigma^\mathrm{H}$, that is, a simple periodic or antiperiodic point of $\Sigma^\mathrm{H}$.  Then the periodic traveling wave $u=f(x-ct)$ does not exhibit any dynamical Hamiltonian-Hopf instability at $\lambda$.
\label{corollary-endpoints-same}
\end{corollary}
\begin{proof}
Since $\nu(\lambda)$ is a simple periodic or antiperiodic point of $\Sigma^\mathrm{H}$, 
$\Delta^\mathrm{H}(\nu(\lambda))^2=4$ while $\Delta_\nu^\mathrm{H}(\nu(\lambda))\neq 0$.  Therefore $F(\nu(\lambda))=0$, but $\nu(\lambda)<0$ as $\lambda$ is nonzero imaginary.
\end{proof}
Corollary~\ref{corollary-endpoints-same} implies that each nonzero periodic or antiperiodic endpoint of a band of $\sigma^\mathrm{H}$ on the imaginary axis has a complex neighborhood in which $\sigma$ and $\sigma^\mathrm{H}$ agree exactly.  Some numerical calculations published in an earlier paper \cite{JonesMMP13} seem to contradict this fact (see in particular the two lower panels\footnote{The lower right-hand panel of Figure~2 of \cite{JonesMMP13} is mislabeled and should instead read ``Superluminal Librational'' as indicated in the caption of the figure.} of Figure~2 of \cite{JonesMMP13}). However, it is now clear that this apparent contradiction was due to poor choice of parameter values that made it appear, at the level of resolution of the computations presented, that unstable spectrum bifurcated from endpoints of bands of stable spectrum.  Choosing different velocities $c$ or energies $E$ results in more representative pictures.  A more representative plot of the type in the lower left-hand panel of Figure~2 of \cite{JonesMMP13} is shown in Figure~\ref{fig:SuperluminalRotational}, and a more representative plot of the type in the lower right-hand panel of Figure~2 of \cite{JonesMMP13} is shown in Figure~\ref{fig:SuperluminalLibrational}.
See \cite{JonesMMP13} for a definition of the energy parameter $E$ that is referred to in the figure captions.
\begin{figure}[h]
\includegraphics{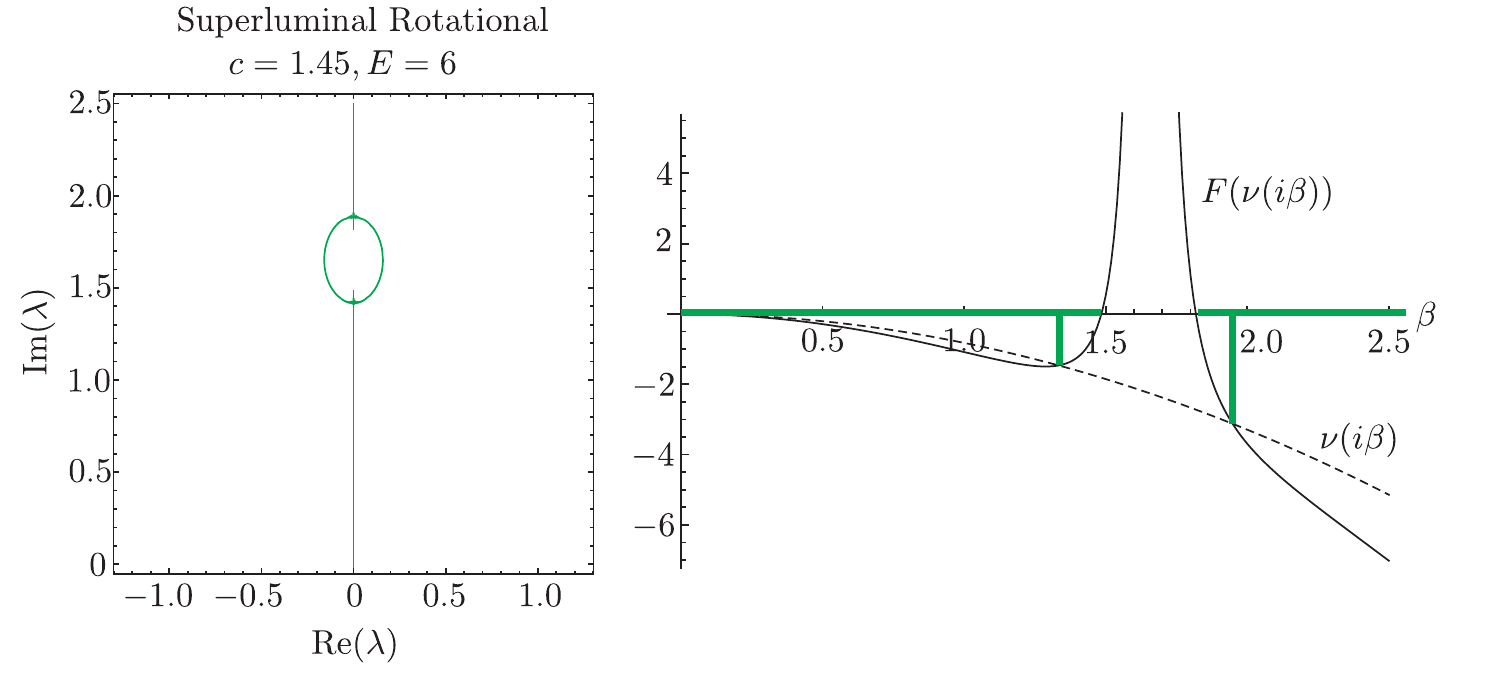}
\caption{Left panel:  the spectrum $\sigma$ in the upper half-plane of a periodic traveling wave for the sine-Gordon equation ($V(u)=1-\cos(u)$) that is of rotational type (that is, $f$ is not periodic although $f'$ is) with energy $E=6$ and wave speed $c=1.45>1$.  Right panel:  plots of $F(\nu(i\beta))$ (solid) and $\nu(i\beta)$ (dashed).  The roots of $F(\nu(i\beta))$ correspond to the endpoints of the bands of $\sigma$ on the imaginary axis, while the roots of $F(\nu(i\beta))-\nu(i\beta)$ correspond to the points of dynamical Hamiltonian-Hopf instability.  Consistent with Corollary~\ref{corollary-endpoints-same}, the dynamical Hamiltonian-Hopf instabilities lie in the \emph{interior} of bands of stable spectrum.}
\label{fig:SuperluminalRotational}
\end{figure}
\begin{figure}[h]
\includegraphics{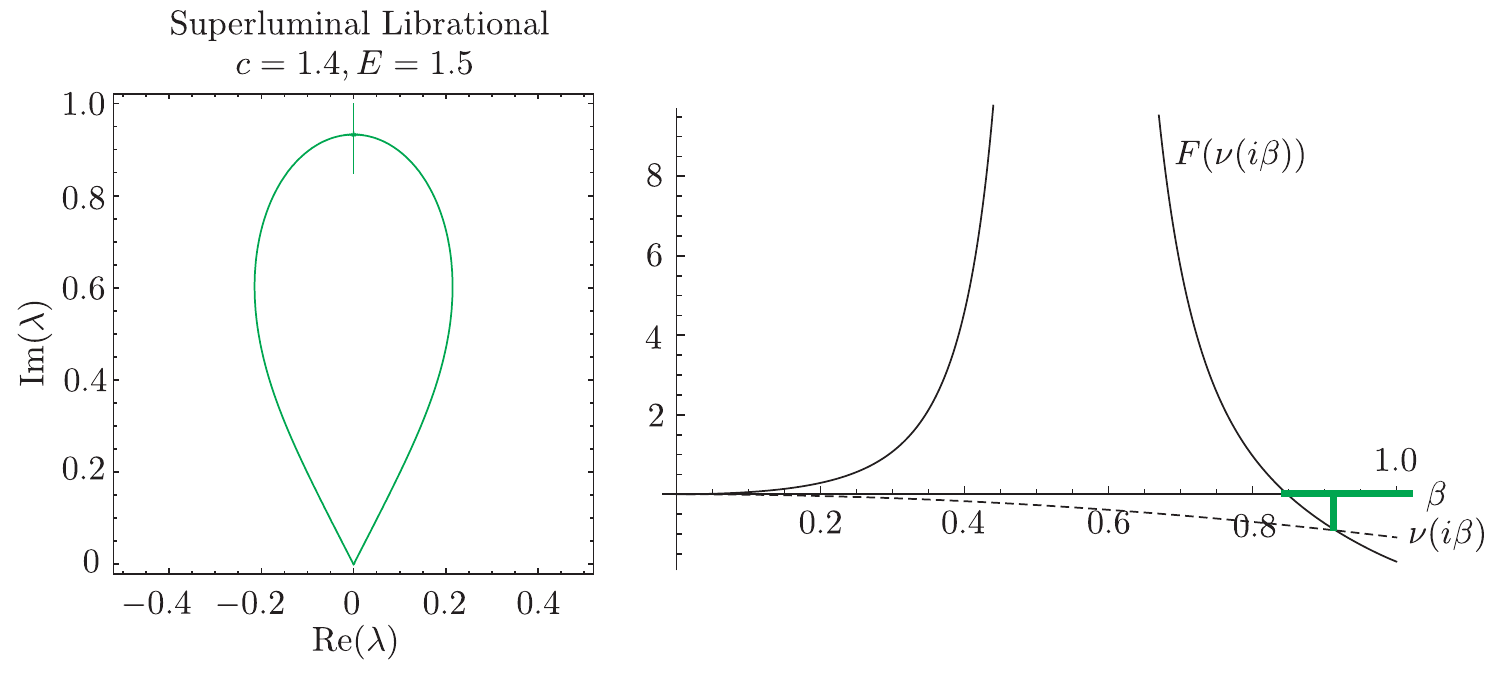}
\caption{As in Figure~\ref{fig:SuperluminalRotational}, but for a periodic traveling wave for the sine-Gordon equation that is of librational type (both $f$ and $f'$ are periodic) with energy $E=1.5$ and wave speed $c=1.4>1$.}
\label{fig:SuperluminalLibrational}
\end{figure}

Other corollaries of Theorem~\ref{theorem-HH-characterize} provide simple conditions under which the existence of a dynamical Hamiltonian-Hopf instability can be guaranteed.  All of the following results are obtained essentially by applying the Intermediate Value Theorem to $F(\nu)-\nu$ for $\nu<0$.  Of course, since $F(\nu)$ blows up at critical points of $\Delta^\mathrm{H}(\cdot)$ that are not double points of $\Sigma^\mathrm{H}$, the hypothesis of continuity is not valid.  However since $F(\nu)$ always blows up with a definite sign, it suffices to consider not the function $F(\nu)-\nu$ but rather the continuous function $\tanh(F(\nu)-\nu)$, the sign of which is exactly the same as that of $F(\nu)-\nu$ for each $\nu<0$.  We may calculate this sign in three simple situations.
\begin{lemma}
Suppose that $\nu<0$ lies in a gap of $\Sigma^\mathrm{H}$. Then $F(\nu)-\nu>0$.
\label{lemma-negative-gap}
\end{lemma}
\begin{proof}
Gaps in the Hill's spectrum $\Sigma^\mathrm{H}$ are characterized by the inequality $\Delta^\mathrm{H}(\nu)^2>4$.  
\end{proof}
\begin{lemma}
The following asymptotic formula holds:
\begin{equation}
F(\nu)-\nu= (c^2-1)\nu + o(\nu\csc(T\sqrt{-\nu})^2),\quad \nu\to -\infty.
\label{eq:F-minus-nu-infinity}
\end{equation}
In particular, for large negative $\nu$ away from the points $\nu=-n^2\pi^2/T^2$, $n\in\mathbb{Z}$, $F(\nu)-\nu$ has the same sign as does $1-c^2$.
\label{lemma-negative-asymptotic}
\end{lemma}
\begin{proof}
This follows from WKB-type asymptotics of solutions of Hill's equation \eqref{eq:Hill} for large negative $\nu$, that is, by approximating \eqref{eq:Hill} by the constant-coefficient equation $y''(z)=\nu y(z)$.  It is not hard to show that the special solutions $y_j(z;\nu)$, $j=1,2$, that are used to construct the monodromy matrix satisfy
\begin{equation}
y_1(z;\nu)=\cos(\sqrt{-\nu}z)+o(1)\quad\text{and}\quad y_2(z;\nu)=\frac{1}{\sqrt{-\nu}}\sin(\sqrt{-\nu}z)+o((-\nu)^{-1/2})
\end{equation}
as well as
\begin{equation}
y_1'(z;\nu)=-\sqrt{-\nu}\sin(\sqrt{-\nu}z)+o((-\nu)^{1/2})\quad\text{and}\quad
y_2'(z;\nu)=\cos(\sqrt{-\nu}z)+o(1),
\end{equation}
all holding in the limit $\nu\to -\infty$ uniformly on the interval $0\le z\le T$.    Therefore $\Delta^\mathrm{H}(\nu)=y_1(T;\nu)+y_2'(T;\nu)=2\cos(\sqrt{-\nu}T)+o(1)$ as $\nu\to -\infty$.  Differentiating 
Hill's equation with respect to $\nu$ and solving for the $\nu$-derivatives of $y_j(z;\nu)$ by the classical method of variation of parameters (as in \cite[Proof of Lemma 2.4]{MagnusW04}), one shows that also $\Delta^\mathrm{H}_\nu(\nu)=T\sin(\sqrt{-\nu}T)/\sqrt{-\nu}+o((-\nu)^{-1/2})$ as $\nu\to -\infty$.  With these ingredients, the claimed asymptotic formula \eqref{eq:F-minus-nu-infinity} follows immediately from \eqref{eq:F-define}.
\end{proof}
\begin{lemma}
If $\Delta_\nu^\mathrm{H}(0)\neq 0$, the following asymptotic formula holds:
\begin{equation}
F(\nu)-\nu=\frac{c^2T^2-\Delta_\nu^\mathrm{H}(0)}{\Delta_\nu^\mathrm{H}(0)}\nu + \mathcal{O}(\nu^2),\quad\nu\uparrow 0.
\label{eq:F-minus-nu-zero-1}
\end{equation}
If instead $\Delta_\nu^\mathrm{H}(0)=0$, then
\begin{equation}
F(\nu)-\nu=\frac{2c^2T^2-\Delta_{\nu\nu}^\mathrm{H}(0)}{\Delta_{\nu\nu}^\mathrm{H}(0)}\nu + \mathcal{O}(\nu^2),\quad \nu\uparrow 0.
\label{eq:F-minus-nu-zero-2}
\end{equation}
\label{lemma-negative-small}
\end{lemma}
\begin{proof}
It is easy to confirm by differentiation of \eqref{eq:KG-traveling-wave} with respect to $z$ that $y(z)=f'(z)$ is a solution of Hill's equation \eqref{eq:Hill} for $\nu=0$, so as $f'(z+T)=f'(z)$, it follows that $\nu=0$ is a periodic point of the Hill's spectrum $\Sigma^\mathrm{H}$, that is, $\Delta^\mathrm{H}(0)=2$.  By \cite[Lemma 2.5]{MagnusW04}, either $\Delta_\nu^\mathrm{H}(0)\neq 0$, or $\Delta_\nu^\mathrm{H}(0)=0$ but $\Delta_{\nu\nu}^\mathrm{H}(0)<0$.  In both cases it is easy to see that the apparent singularity in $F(\nu)$ at $\nu=0$ is removable, and the formulae \eqref{eq:F-minus-nu-zero-1} and \eqref{eq:F-minus-nu-zero-2} follow by Taylor expansion.
\end{proof} 
In \cite{JonesMMP14}, the sign $\modndx:=\sgn(\Delta_\nu^\mathrm{H}(0))$ (defined as $\modndx=0$ if $\Delta_\nu^\mathrm{H}(0)=0$) is called the \emph{modulational instability index} of the periodic traveling wave $u=f(x-ct)$.  If $\modndx=1$ then generically the spectrum $\sigma$ is tangent to (and possibly locally coincides with) the imaginary axis at the origin $\lambda=0$, while if $\modndx=-1$ then there are curves of unstable spectrum meeting the origin making nonzero angles with the imaginary axis (this latter condition is what characterizes a strong modulational instability of the periodic traveling wave $u=f(x-ct)$).  The sign of $c^2T^2-\Delta_\nu^\mathrm{H}(0)$ also carries meaning.
To see this, consider the function $D(\lambda,\mu):=\det(\mathbf{M}(\lambda)-\mu\mathbb{I})$, whose roots in the $\lambda$-plane given a fixed unit-modulus value $\mu=e^{i\theta}$, $\theta\in\mathbb{R}$, are points of the spectrum $\sigma$ corresponding to a Floquet multiplier $\mu(\lambda)=e^{i\theta}$.  In particular, the function $D(\lambda,1)$ is called the \emph{periodic Evans function} as it is an entire function whose roots are the periodic points of the spectrum $\sigma$.  By Abel's Theorem, $\det(\mathbf{M}(\lambda))=e^{2\lambda cT/(c^2-1)}$, and Scott's substitution shows that $\tr(\mathbf{M}(\lambda))=e^{\lambda c T/(c^2-1)}\tr(\mathbf{M}^\mathrm{H}(\nu(\lambda)))=e^{\lambda cT/(c^2-1)}\Delta^\mathrm{H}(\nu(\lambda))$.  Therefore, by Taylor expansion,
\begin{equation}
D(\lambda,1)=\frac{c^2T^2-\Delta_\nu^\mathrm{H}(0)}{(c^2-1)^2}\lambda^2 + \mathcal{O}(\lambda^3),\quad\lambda\to 0,
\end{equation}
so it follows that $\sgn(c^2T^2-\Delta_\nu^\mathrm{H}(0))=\sgn(D_{\lambda\lambda}(0,1))$.
In \cite{JonesMMP14}, the sign defined by $\parndx:=\sgn((c^2-1)D_{\lambda\lambda}(0,1))$ in the generic case that $D_{\lambda\lambda}(0,1)\neq 0$
is called the \emph{parity index} of the periodic traveling wave $u=f(x-ct)$.  The condition 
$\parndx=-1$ guarantees the existence of a positive real (and by Hamiltonian symmetry, a negative real) periodic point in the spectrum $\sigma$ (more precisely, this condition indicates that the number of positive roots of $D(\lambda,1)$ counted with multiplicity is odd).  These indices also can indicate the presence of dynamical Hamiltonian-Hopf instabilities as the following simple corollaries of Theorem~\ref{theorem-HH-characterize} show.

\begin{corollary}
In the generic case that $\Delta_\nu^\mathrm{H}(0)\neq c^2T^2$ (i.e., $D_{\lambda\lambda}(0,1)\neq 0$), the periodic traveling wave $u=f(x-ct)$ exhibits a dynamical Hamiltonian-Hopf instability if $\modndx\parndx=-1$, where $\modndx$ is the modulational instability index and 
$\parndx$ is the parity index.  
\label{corollary-index-product}
\end{corollary}
\begin{proof}
By Lemma~\ref{lemma-negative-asymptotic} and \eqref{eq:F-minus-nu-zero-1} of Lemma~\ref{lemma-negative-small}, the continuous function $\tanh(F(\nu)-\nu)$ has opposite signs for small and large negative $\nu$ if $\modndx\parndx=-1$, and hence there is some $\nu<0$ for which $\nu=F(\nu)$.
\end{proof}
Corollary~\ref{corollary-index-product} shows that a dynamical Hamiltonian-Hopf instability always occurs somewhere on the imaginary axis if either there is a strong modulational instability at the origin, or there exist an odd number of positive periodic points of $\sigma$ (but not both).  This result is consistent with the intuition that a curve of spectrum emanating from the origin that does not intersect the real axis should intersect the imaginary axis, while a curve of spectrum emanating from the real axis that does not return to the real axis or meet the origin should intersect the imaginary axis.  A result partially complementary to Corollary~\ref{corollary-index-product} is the following.
\begin{corollary}
In the generic case that $\Delta_\nu^\mathrm{H}(0)\neq c^2T^2$ (i.e., $D_{\lambda\lambda}(0,1)\neq 0$) but $\modndx\parndx=1$, the 
periodic traveling wave $u=f(x-ct)$ nonetheless exhibits a dynamical Hamiltonian-Hopf instability if there is a negative gap in the Hill's spectrum $\Sigma^\mathrm{H}$ and $c^2-1>0$.
\label{corollary-blowup-in-gaps}
\end{corollary}
\begin{proof}
Since $\modndx\parndx=1$, $F(\nu)-\nu$ has the same sign for large and small negative $\nu$.  However, from Lemma~\ref{lemma-negative-asymptotic}, the condition $c^2-1>0$ ensures that
$F(\nu)-\nu<0$ holds for some (large) negative $\nu$, while from Lemma~\ref{lemma-negative-gap}, the existence of a negative gap in $\Sigma^\mathrm{H}$ ensures that $F(\nu)-\nu>0$ holds for some $\nu<0$.  Applying the Intermediate Value Theorem to the continuous function $\tanh(F(\nu)-\nu)$ and using Theorem~\ref{theorem-HH-characterize} then completes the proof.
\end{proof}
Finally, we have the following.
\begin{corollary}
A superluminal (i.e., $c^2>1$) periodic traveling wave $u=f(x-ct)$ has at least two points of  dynamical Hamiltonian-Hopf instability on the positive imaginary $\lambda$-axis for each open gap of $\Sigma^\mathrm{H}$ with $\nu$ sufficiently negative.
\end{corollary}
\begin{proof}
This is a consequence of Lemma~\ref{lemma-negative-gap} and Lemma~\ref{lemma-negative-asymptotic}.  The gaps become very small as $\nu\to -\infty$, but each open gap contains a point at which $\tanh(F(\nu)-\nu)=1$, while on either side of each of these points the asymptotic formula \eqref{eq:F-minus-nu-infinity} shows that $\tanh(F(\nu)-\nu)$ is negative if $c^2>1$.  Therefore, there are two roots of $\nu=F(\nu)$ near each (small) open gap in $\Sigma^\mathrm{H}$ if $\nu<0$ is large enough in magnitude, with one root lying in each surrounding band of $\Sigma^\mathrm{H}$, and the desired result follows from Theorem~\ref{theorem-HH-characterize}.
\end{proof}

\subsection{Acknowledgements}
R. Marangell gratefully acknowledges the partial support of the Australian Research Council on grant DP110102775.
P. D. Miller was supported by the National Science Foundation under grant number DMS-1206131, by a Research Fellowship from the Simons Foundation, and is grateful for the additional support of the University of Sydney under an International Research Collaboration Award that sponsored a long-term visit to Sydney during which time this work was completed.

\section{Local analysis of the spectrum near the imaginary axis}
\label{sec:local-analysis}
Here we give the proof of Theorem~\ref{theorem-HH-characterize}.  We will study the spectrum $\sigma$ near a point $\lambda=i\beta$ with $\beta>0$ (the case of $\beta<0$ follows automatically by the real symmetry $\sigma^*=\sigma$, see \cite[Section 3.3.1]{JonesMMP14}).
Since $\sigma$ is closed, it suffices to suppose that $i\beta\in\sigma$, or equivalently, since $\sigma\cap i\mathbb{R}=\sigma^\mathrm{H}\cap i\mathbb{R}$, $i\beta\in\sigma^\mathrm{H}$.
The argument splits into two cases, depending on whether $\lambda=i\beta$ is at the edge of a band of spectrum or not.
\subsection{The case that $\nu(i\beta)\in\Sigma^\mathrm{H}$ is a periodic or antiperiodic eigenvalue of Hill's equation}
In this case, $\lambda=i\beta$ is a simple root of $\Delta^\mathrm{H}(\nu(\lambda))^2-4$, and in particular, $\Delta^\mathrm{H}(\nu(i\beta))=\pm 2$.  Thus, as $\lambda\to i\beta$, 
\begin{equation}
\Delta^\mathrm{H}(\nu(\lambda))=\pm 2 + it_1(\lambda-i\beta) + \mathcal{O}((\lambda-i\beta)^2),
\end{equation}
where
\begin{equation}
t_1:=\left.-i\frac{d}{d\lambda}\Delta^\mathrm{H}(\nu(\lambda))\right|_{\lambda=i\beta}=
\frac{2\beta}{(c^2-1)^2}\Delta_\nu^\mathrm{H}(\nu(i\beta))
\end{equation}
is a real number (nonzero by assumption).  By the quadratic formula, the two Hill Floquet multipliers, which are the roots $\mu^\mathrm{H}$ of $\left(\mu^{\mathrm{H}}\right)^{2}-\Delta^\mathrm{H}\mu^\mathrm{H}+1=0$, have the local representations
\begin{equation}
\begin{split}
\mu^\mathrm{H}(\nu(\lambda))&=\pm 1 +i(\mp i t_1 (\lambda-i\beta))^{1/2}+\mathcal{O}(\lambda-i\beta)\\
\mu^\mathrm{H}(\nu(\lambda))&=\pm 1 -i(\mp i t_1(\lambda-i\beta))^{1/2}+\mathcal{O}(\lambda-i\beta).
\end{split}
\end{equation}
Scott's substitution implies that the corresponding Floquet multipliers of \eqref{eq:spectral-problem}
are given by 
\begin{equation}
\mu(\lambda)=e^{\lambda cT/(c^2-1)}\mu^\mathrm{H}(\nu(\lambda)).  
\label{eq:Scott-multipliers}
\end{equation}
Since the exponential factor is differentiable at $\lambda=i\beta$, it follows that the local representations of the two multipliers $\mu(\lambda)$ are 
\begin{equation}
\begin{split}
\mu(\lambda)&=\pm e^{i\beta cT/(c^2-1)} + ie^{i\beta cT/(c^2-1)}(\mp it_1 (\lambda-i\beta))^{1/2}+\mathcal{O}(\lambda-i\beta)\\
\mu(\lambda)&=\pm e^{i\beta cT/(c^2-1)} - ie^{i\beta cT/(c^2-1)}(\mp it_1 (\lambda-i\beta))^{1/2}+\mathcal{O}(\lambda-i\beta)
\end{split}
\end{equation}
The spectrum $\sigma$ is characterized by the condition $|\mu(\lambda)|=1$ for one or the other of the above two multipliers.  Therefore to have $\lambda\in\sigma$ it is necessary that $\mathrm{Im}((\mp it_1(\lambda-i\beta))^{1/2})=\mathcal{O}(\lambda-i\beta)$ as $\lambda\to i\beta$.  This, in turn, 
requires that $\sigma$ be locally tangent to the imaginary axis at $\lambda=i\beta$, but only in the direction that $\mp it_1(\lambda-i\beta)>0$, that is, in the direction of the band of $\sigma$ on the imaginary axis emanating from $\lambda=i\beta$.  But the latter purely imaginary spectrum already exhausts the full multiplicity of $\sigma$, and hence $\sigma$ agrees exactly with $\sigma^\mathrm{H}$ in a complex neighborhood of $\lambda=i\beta$ and therefore is purely imaginary.  There is no dynamical Hamiltonian-Hopf instability.  This is actually an independent proof of Corollary~\ref{corollary-endpoints-same};  it remains to show consistency with Theorem~\ref{theorem-HH-characterize}.  But consistency follows immediately because at all simple roots of $\Delta^\mathrm{H}(\nu)^2-4$ we have $F(\nu)=0$, but $\nu=\nu(i\beta)<0$ by the hypothesis that $\beta>0$.  

\subsection{The case that $\nu(i\beta)$ is in the interior of $\Sigma^\mathrm{H}$} 
If on the contrary $\lambda=i\beta$ is in the interior of a band of $\sigma$ on the imaginary axis, then either the strict inequality $\Delta^\mathrm{H}(\nu(i\beta))^2<4$ holds, or we have a double point of spectrum and instead $\Delta^\mathrm{H}(\nu(i\beta))^2=4$ and $\Delta_\nu^\mathrm{H}(\nu(i\beta))=0$ but $\Delta^\mathrm{H}_{\nu\nu}(\nu(i\beta))\neq 0$, having the opposite sign to
$\Delta^\mathrm{H}(\nu(i\beta))=\pm 2$ \cite[Lemma 2.5]{MagnusW04}.  In both cases, the two Hill Floquet multipliers $\mu^\mathrm{H}(\nu(\lambda))$ may be considered as locally analytic functions
of $\lambda$ because the quadratic discriminant $4-\Delta^\mathrm{H}(\nu(\lambda))^2$ is analytic and either strictly positive at $\lambda=i\beta$ or nonnegative and vanishing to precisely second order at $\lambda=i\beta$.  

In the former case, the locally distinct analytic Hill Floquet multipliers have Taylor expansions about $\lambda=i\beta$ of the form
\begin{equation}
\begin{split}
\mu^\mathrm{H}(\nu(\lambda))&=\frac{1}{2}\left(\Delta^\mathrm{H}(\nu(\lambda))\pm i\sqrt{4-\Delta^\mathrm{H}(\nu(\lambda))^2}\right)\\
&=\frac{1}{2}\left(t_0 \pm i\sqrt{4-t_0^2}\right) + \frac{1}{2}\left(it_1 \pm\frac{t_0t_1}{\sqrt{4-t_0^2}}\right)(\lambda-i\beta) + \mathcal{O}((\lambda-i\beta)^2),
\end{split}
\label{eq:muH-expansion}
\end{equation}
where the corresponding Taylor expansion of the Hill discriminant is
\begin{equation}
\Delta^\mathrm{H}(\nu(\lambda)))=t_0+it_1(\lambda-i\beta)+ \mathcal{O}((\lambda-i\beta)^2),\quad\lambda\to i\beta,
\end{equation}
with
\begin{equation}
t_0=\Delta^\mathrm{H}(\nu(i\beta))\in (-2,2),\quad t_1=\frac{2\beta}{(c^2-1)^2}\Delta^{\mathrm{H}}_\nu(\nu(i\beta))\in\mathbb{R}.
\label{eq:t0-t1-express}
\end{equation}
In the second line of \eqref{eq:muH-expansion}, the (strictly, since $t_0\in (-2,2)$) positive square root is meant.  This implies that there is a real angle $\theta_0\in (0,\pi)$ such that
\begin{equation}
\frac{1}{2}\left(t_0\pm i\sqrt{4-t_0^2}\right)=e^{\pm i\theta_0},
\end{equation}
and hence \eqref{eq:muH-expansion} could also be written simply as
\begin{equation}
\mu^\mathrm{H}(\nu(\lambda))=e^{\pm i\theta_0} \left(1\pm\frac{t_1}{\sqrt{4-t_0^2}}(\lambda-i\beta)+\mathcal{O}((\lambda-i\beta)^2)\right).
\label{eq:muH-expansion-rewrite}
\end{equation}

The Floquet multipliers for the linearized Klein-Gordon spectral problem \eqref{eq:spectral-problem}
are related explicitly to those of Hill's equation \eqref{eq:Hill} by the identity \eqref{eq:Scott-multipliers}.
Expanding the exponential factor about $\lambda=i\beta$ and using \eqref{eq:muH-expansion-rewrite} yields the corresponding Taylor expansions of the analytic functions representing the two multipliers $\mu(\lambda)$ in the form
\begin{equation}
\begin{split}
\mu(\lambda)&=e^{i\beta cT/(c^2-1)}\left(1+\frac{cT}{c^2-1}(\lambda-i\beta)+\mathcal{O}((\lambda-i\beta)^2)\right)\mu^\mathrm{H}(\nu(\lambda))\\
&=e^{i[\beta cT/(c^2-1)\pm \theta_0]}\left(1 + \delta_\pm(\lambda-i\beta)+\mathcal{O}((\lambda-i\beta)^2)\right),
\end{split}
\end{equation}
where
\begin{equation}
\delta_\pm:=\frac{cT}{c^2-1}\pm \frac{t_1}{\sqrt{4-t_0^2}}.
\end{equation}

If $\delta_\pm$ are both nonzero, then both multipliers are locally univalent near $\lambda=i\beta$ and hence are locally invertible analytic functions.  Consequently the equation $\mu(\lambda)=e^{i\theta}$, $\theta\in\mathbb{R}$ will be solvable for two analytic functions $\lambda=\lambda(\theta)$ in the vicinity of $\theta=\beta cT/(c^2-1)\pm\theta_0$.  Since $\delta_\pm$ are real, it is clear that these functions describe two smooth curves of spectrum $\sigma$ that are tangent to the imaginary axis at $\lambda=i\beta$.  Since the imaginary spectrum present on the imaginary axis near $\lambda=i\beta$ already exhausts the full multiplicity of $\sigma$, the smooth curves $\lambda=\lambda(\theta)$ are purely imaginary near $\lambda=i\beta$.

On the other hand, if either $\delta_+=0$ or $\delta_-=0$, then $\mu'(i\beta)=0$ for one of the multipliers and therefore this multiplier will not be invertible, making the spectrum $\sigma$ locally more complicated.  In the generic case that $\mu''(i\beta)\neq 0$, the equation $\mu(\lambda)=e^{i\theta}$ involving the corresponding multiplier will give rise to two curves of $\sigma$ that cross perpendicularly at $\lambda=i\beta$, one of which is tangent to the imaginary axis.  The same equation involving the other (invertible) multiplier will give rise to a single curve of $\sigma$ tangent to the imaginary axis at $\lambda=i\beta$.  
More generally, if the first nonzero derivative of the noninvertible multiplier is $\mu^{(p)}(i\beta)\neq 0$,
$p>1$, there will be $p$ smooth curves of $\sigma$ crossing at $\lambda=i\beta$ with equal angles, exactly one of which is tangent to the imaginary axis, and from the invertible multiplier there will be an additional curve of $\sigma$ tangent to the imaginary axis at $\lambda=i\beta$.  The two curves of spectrum locally tangent to the imaginary axis must in fact coincide locally with the imaginary axis, by the same argument as in the locally univalent case.

The condition for the existence at $\lambda=i\beta$ of a (strong, because the unstable spectrum consists of curves transverse to the imaginary axis) dynamical Hamiltonian-Hopf instability is therefore
\begin{equation}
\delta_+\delta_-=\frac{c^2T^2}{(c^2-1)^2}-\frac{t_1^2}{4-t_0^2}=0.
\end{equation}
Using \eqref{eq:t0-t1-express} allows $(c^2-1)^2\delta_+\delta_-$ to be rewritten as
\begin{equation}
(c^2-1)^2\delta_+\delta_-=c^2T^2+\frac{4\nu\Delta^{\mathrm{H}}_\nu(\nu)^2}{4-\Delta^\mathrm{H}(\nu)^2},\quad\nu=\nu(i\beta)=-\frac{\beta^2}{(c^2-1)^2}<0.
\label{eq:condition-rewrite}
\end{equation}
Since by assumption on $\beta$ we have $\Delta^\mathrm{H}(\nu)^2<4$ and $\Delta^{\mathrm{H}}_\nu(\nu)\neq 0$, the condition $\delta_+\delta_-=0$ is equivalent to the condition $\nu=F(\nu)$ with $\nu=\nu(i\beta)$ and $F$ defined by \eqref{eq:F-define}.

It remains to consider the case that $\lambda=i\beta$ is a double point of $\sigma^\mathrm{H}$.
At a double point we have the Taylor expansion of the Hill discriminant in the form
\begin{equation}
\Delta^\mathrm{H}(\nu(\lambda))=\pm 2 \pm 2t_2^2(\lambda-i\beta)^2 + \mathcal{O}((\lambda-i\beta)^3),\quad\lambda\to i\beta,
\label{eq:double-point-discriminant}
\end{equation}
where
\begin{equation}
t_2 :=\frac{\beta}{(c^2-1)^2}\sqrt{ \mp 2\Delta^\mathrm{H}_{\nu\nu}(\nu(i\beta))}>0.
\end{equation}
It follows that the two analytic Hill Floquet multipliers are
\begin{equation}
\begin{split}
\mu^\mathrm{H}(\nu(\lambda))&=\pm 1 +t_2(\lambda-i\beta) +\mathcal{O}((\lambda-i\beta)^2)\\
\mu^\mathrm{H}(\nu(\lambda))&=\pm 1-t_2(\lambda-i\beta)+\mathcal{O}((\lambda-i\beta)^2).
\end{split}
\end{equation}
Using \eqref{eq:Scott-multipliers}, we then have
\begin{equation}
\begin{split}
\mu(\lambda)&=\pm e^{i\beta cT/(c^2-1)}\left(1+\hat{\delta}_+(\lambda-i\beta)+\mathcal{O}((\lambda-i\beta)^2)\right)\\
\mu(\lambda)&=\pm e^{i\beta cT/(c^2-1)}\left(1+\hat{\delta}_-(\lambda-i\beta)+\mathcal{O}((\lambda-i\beta)^2)\right),
\end{split}
\end{equation}
where
\begin{equation}
\hat{\delta}_\pm:=\frac{cT}{c^2-1}\pm t_2.
\end{equation}
As before, there will exist a (strong) dynamical Hamiltonian-Hopf instability at the double point $\lambda=i\beta$
if and only if $\hat{\delta}_+\hat{\delta}_-=0$.  Comparing with \eqref{eq:condition-rewrite}, it is clear that to complete the proof one must simply check that whenever $i\beta$ is a double point of $\sigma^\mathrm{H}$, 
\begin{equation}
\lim_{\lambda\to i\beta}\frac{4\nu(\lambda)\Delta_\nu^\mathrm{H}(\nu(\lambda))^2}{4-\Delta^\mathrm{H}(\nu(\lambda))^2}=-(c^2-1)^2t_2^2.
\label{eq:limit-expression}
\end{equation}
However, using the fact that 
\begin{equation}
4\nu(\lambda)\Delta_\nu^\mathrm{H}(\nu(\lambda))^2=(c^2-1)^2\frac{d}{d\lambda}\Delta^\mathrm{H}(\nu(\lambda)),
\end{equation}
it is easy to confirm that \eqref{eq:limit-expression} holds as a consequence of
\eqref{eq:double-point-discriminant}. This concludes the proof of Theorem \ref{theorem-HH-characterize}. \qed

\end{document}